\documentclass[10pt]{amsart}
\usepackage{amssymb}
\usepackage{amsrefs}
\usepackage[french,english]{babel}
\usepackage[T1]{fontenc}
\usepackage{xcolor}
\usepackage[hypertexnames=false]{hyperref}

\definecolor{linkblue}{RGB}{35,80,130}
\hypersetup{
  colorlinks=true,
  allcolors=linkblue
}

\renewcommand{\eprint}[1]{%
  \href{https://arxiv.org/abs/#1}{\nolinkurl{arXiv:#1}}}
\renewcommand{\PrintDOI}[1]{%
  \href{https://doi.org/#1}{DOI \nolinkurl{#1}}}

\title[Coordinate systems and maximal order complexity]{Coordinate systems and maximal order complexity in Banach lattices}

\author{A.~Avil\'es}
\address{Universidad de Murcia, Departamento de Matem\'aticas, Campus de Espinardo 30100 Murcia, Spain.}
\email{avileslo@um.es}

\author{C.~Rosendal}
\address{Department of Mathematics\\University of Maryland\\4176 Campus Drive\\William E. Kirwan Hall\\College Park, MD 20742-4015\\USA}
\email{rosendal@umd.edu}
\urladdr{https://sites.google.com/view/christian-rosendal/}

\author[M.~A.~Taylor]{M.~A.~Taylor}
\address{Mathematical Institute, University of Oxford, Andrew Wiles Building, Radcliffe Observatory Quarter, Woodstock Road, Oxford, OX2 6GG, United Kingdom}
\email{mitchell.taylor@maths.ox.ac.uk}

\author{P.~Tradacete}
\address{Instituto de Ciencias Matem\'aticas (CSIC-UAM-UC3M-UCM)\\
Consejo Superior de Investigaciones Cient\'{\i}ficas\\
C/ Nicol\'as Cabrera, 13--15, Campus de Cantoblanco UAM\\
28049 Madrid, Spain.}
\email{pedro.tradacete@icmat.es}
\urladdr{https://www.icmat.es/miembros/ptradacete/}

\newcommand{\norm}[1]{\lVert#1\rVert}
\newcommand{\Norm}[1]{\big\lVert#1\big\rVert}
\newcommand{\BNORM}[1]{\Bigg\lVert#1\Bigg\rVert}
\newcommand{\triple}[1]{|\!|\!|#1|\!|\!|}

\newcommand{\Mgd}[2]{\big\{\,{#1}\;\big|\;{#2}\,\big\}}
\newcommand{\conv}[2]{\mathop{\underset{#2}{\overset{#1}\longrightarrow}}}

\newcommand{\N}{\mathbb N}
\newcommand{\R}{\mathbb R}

\newcommand{\eps}{\varepsilon}
\newcommand{\ex}{\exists}
\newcommand{\all}{\forall}
\newcommand{\del}{\;\big|\;}
\newcommand{\equi}{\Leftrightarrow}
\newcommand{\spa}{\mathsf{span}}
\newcommand{\ku}{\mathcal}
\newcommand{\ov}{\overline}
\newcommand{\tom}{\emptyset}

\newcommand{\Tr}{\mathsf{Tr}_{\N}}
\newcommand{\WF}{\mathsf{WF}_{\N}}
\newcommand{\concat}{^{\smallfrown}}
\newcommand{\last}{\mathsf{last}}

\newcommand{\leanstatement}[3]{%
  \href{https://github.com/pedrotradacete/DSTBL3/blob/47f515ee4b565fc55e91754f348d4e14a43f0de4/#1\#L#2}%
  {\ensuremath{\forall}\,\nolinkurl{#3}}}
\newcommand{\leanformalization}[1]{%
  {\normalfont\footnotesize [#1]}}

\theoremstyle{plain}
\newtheorem{thm}{Theorem}[section]
\newtheorem*{theorem*}{Theorem}

\newtheorem{lemme}[thm]{Lemma}
\newtheorem{prop} [thm] {Proposition}
\newtheorem{defi} [thm] {Definition}

\newtheorem{theorem}{Theorem}

\theoremstyle{definition}
\newtheorem{rem}[thm]{Remark}

\makeatletter
\@ifundefined{subjclassname@2020}{%
  \@namedef{subjclassname@2020}{\textup{2020} Mathematics Subject Classification}%
}{}
\makeatother

\thanks{C.~Rosendal~was partially supported by the U.S.~National Science Foundation under Award Numbers DMS-2246986 and DMS-2204849. A.~Avil\'{e}s was supported by grant PID2021-122126NB-C32 funded by MICIU/AEI/10.13039/501100011033 and by ERDF, ``A way of making Europe''. A.~Avil\'{e}s and P.~Tradacete were supported by Fundaci\'{o}n S\'{e}neca--ACyT Regi\'{o}n de Murcia. P.~Tradacete was partially supported by grants PID2024-162214NB-I00 and CEX2023-001347-S funded by MCIN/AEI/10.13039/501100011033.}

\begin{document}
\subjclass[2020]{Primary: 46B42, 46B15, Secondary: 03E15, 54A20}
\keywords{Banach lattices, order bases, Schauder bases, order convergence, complete coanalytic sets}

\begin{abstract}
We resolve three questions arising from the study of coordinate systems and order convergence in Banach lattices. Firstly, every order (respectively, $\sigma$-order) basis $(e_n)$ for a Banach lattice $X=[e_n]$ is automatically a Schauder basis. In fact, the associated partial sum projections satisfy a maximal inequality, and hence every such basis is uniform.

Secondly, we show that there is a separable Banach lattice $X$ in which the set
\[
X_{\downarrow0}=\Mgd{(x_n)\in X^\N}{x_1\geqslant x_2\geqslant\ldots\geqslant 0=\inf_nx_n}
\]
is $\mathbf{\Pi}^1_1$-complete. Finally, we construct a separable Banach lattice $Y$ for which the set $Y^{\uparrow\infty}$ of increasing positive  sequences which are not order bounded is $\mathbf{\Pi}^1_1$-complete. It follows that neither $X_{\downarrow0}$ nor $Y^{\uparrow\infty}$ can be described in any simpler way than by the nonexistence of positive lower and  upper bounds, respectively.
\end{abstract}

\maketitle

\section{Introduction}

This paper settles three questions left open in \cites{ARTT,HOF}.  The first concerns coordinate systems in Banach lattices.  If $C$ denotes norm, weak, uniform, $\sigma$-order or order convergence, a sequence $(e_n)$ in a Banach lattice $X$ is called a $C$-basis when every $x\in X$ admits a unique expansion
\[
x=^C\sum_{n=1}^{\infty} a_ne_n,
\]
where the infinite series is $C$-convergent.
For norm convergence this is the usual notion of a Schauder basis, introduced by Schauder in \cite{schauder}. In Banach spaces, the continuity of the coordinate functionals follows automatically from the classical basis theorem \cite{Banach}*{p.~111}.  For order and $\sigma$-order bases, however, continuity of the associated coordinate functionals is not a priori clear, as the partial sums are only required to converge in the lattice order.

 Assuming analytic determinacy, it was proved in \cite{ARTT} that every $\sigma$-order basis $(e_n)$ for a Banach lattice $X=[e_n]$ is a uniform basis and hence a Schauder basis.  The additional set-theoretic hypothesis entered through definability.  In separable Banach lattices, order and $\sigma$-order convergence of sequences coincide, but the relation $x_n\conv{\sigma\sf o}{n}x$ is, a priori, only $\mathbf{\Delta}^1_2$.  Thus, the Baire category argument used to prove continuity of the coordinate map required projective regularity.  Our first result removes this issue entirely.

\begin{theorem}\label{theoremA}
\leanformalization{
  \leanstatement{TheoremA/Lean_Code/ThmA/MaximalInequality.lean}{244}{theoremA_sigma_order_form},
  \leanstatement{TheoremA/Lean_Code/ThmA/MaximalInequality.lean}{266}{theoremA_order_form},
  \leanstatement{TheoremA/Lean_Code/ThmA/MaximalInequality.lean}{257}{theoremA_sigma_order_schauder}, and
  \leanstatement{TheoremA/Lean_Code/ThmA/MaximalInequality.lean}{276}{theoremA_order_schauder}}
Let $X$ be a Banach lattice and let $(e_n)$ be a sequence in $X$ with $X=[e_n]$. If $(e_n)$ is either an order basis or a $\sigma$-order basis for $X$, then $(e_n)$ is a uniform basis for $X$. More precisely, there is a constant $C<\infty$ such that, for every $N\geqslant1$ and all scalars $a_1,\ldots,a_N$, we have
\begin{equation}\label{eq:maxineq}
\BNORM{\bigvee_{m=1}^N\left|\sum_{k=1}^ma_ke_k\right|}
\leqslant C\Norm{\sum_{k=1}^Na_ke_k}.
\end{equation}
\end{theorem}

The proof is based on a lattice analogue of Mazur's proof that weak bases are Schauder bases.  For a weak basis one completes the argument by applying the inverse mapping theorem to the norm $\triple x=\sup_m\norm{P_mx}$.  Here, $\sigma$-order convergence gives order boundedness of the partial sums, and the corresponding norm is
\[
\rho(x)=\inf\Mgd{\norm u}{u\in X_+\text{ and } |P_mx|\leqslant u\text{ for all }m}.
\]
The main point is to prove that $(X,\rho)$ is complete.  The inverse mapping theorem then yields $\rho\leqslant C\norm{\cdot}$, which implies the above maximal inequality.  Thus, automatic continuity is obtained in ZFC without lowering the descriptive complexity of $\sigma$-order convergence.

The remaining questions are descriptive-set-theoretic in nature.  In \cite{HOF}, the complexity of $\sigma$-order convergence was analyzed through
\[
X_{\downarrow0}=\Mgd{(x_n)\in X^\N}{x_1\geqslant x_2\geqslant\ldots\geqslant 0=\inf_nx_n}.
\]
For every separable Banach lattice $X$, this set is coanalytic.  Moreover, in \cite{HOF} a Fatou hierarchy  was introduced so that $X_{\downarrow0}$ is Borel exactly when $X$ is $\alpha$-Fatou for some countable ordinal $\alpha$. It was further shown that $X_{\downarrow0}$ is Borel precisely when $\sigma$-order convergence is analytic.  Spaces of arbitrarily high countable Fatou rank were constructed in \cite{HOF}, but it remained open whether a single separable Banach lattice could achieve the coanalytic upper bound. Our next theorem answers this question affirmatively.

\begin{theorem}\label{theoremB}
\leanformalization{\leanstatement{ThmBC/Lower/Complete.lean}{441}{lower_complete_template}}
There is a separable Banach lattice $X$ such that $X_{\downarrow0}$ is $\mathbf{\Pi}^1_1$-complete.
\end{theorem}

There is a parallel question for increasing sequences. Put
\[
X^{\uparrow\infty}=\Mgd{(x_n)\in X^\N}{0\leqslant x_1\leqslant x_2\leqslant\ldots\text{ and $(x_n)$ is not order bounded}}.
\]
This set is again coanalytic in every separable Banach lattice, while the criteria in \cite{HOF} give many natural hypotheses under which it is Borel.  Our final result shows that the trivial upper bound on the complexity of $X^{\uparrow\infty}$ is, nevertheless, sharp.

\begin{theorem}\label{theoremC}
\leanformalization{\leanstatement{ThmBC/Upper/Upper.lean}{350}{upper_complete_template}}
There is a separable Banach lattice $Y$ for which $Y^{\uparrow\infty}$ is $\mathbf{\Pi}^1_1$-complete.
\end{theorem}

The proofs of Theorems~B and C  amalgamate the examples in \cite{HOF} using an explicit weighted tree.  A tree is sent to a sequence obtained by appropriately attaching to the tree  either a decreasing sequence with infimum zero or an increasing sequence that is not order bounded.  Well-foundedness and ill-foundedness are then read off from the absence or presence of the relevant order bounds.  Consequently, the natural reductions from well-founded trees witness complete coanalyticity.

The above results illustrate once again the difference between general order convergent sequences and those witnessing the convergence of a basis expansion. See \cites{gumenchuk,taylor} for the first examples of this kind and \cites{berna2026maximal,yu2026refined} for further developments. 

The paper is organized as follows. Section~\ref{sec:prelim} recalls the convergence and basis notions and records the elementary descriptive-set-theoretic facts. Section~\ref{sec:bases} proves Theorem~A. Section~\ref{sec:tree-template} develops the common weighted tree construction. Sections~\ref{sec:lower-complete} and~\ref{sec:upper-complete} prove Theorems~B and C, respectively.

\subsection{Lean formalization}

The results in this paper have been formalized in Lean~4 using Mathlib and the BanLat library. Every proposition, theorem, and lemma below includes a link, marked by $\forall$, to the corresponding declaration at a fixed revision of \href{https://github.com/pedrotradacete/DSTBL3}{\texttt{pedrotradacete/DSTBL3}}. When a paper statement is represented by several declarations, all of them are listed. For Theorem~A, the links cover the maximal inequality and the resulting Schauder basis statements; the final uniform basis implication uses the external theorem cited in its proof.

%%%%%%%%%%%%%%%%%%%%%%%%%%%%%%%%%%%%%%%%%%%%
%%%%%%%%%%%%%%%%%%%%%%%%%%%%%%%%%%%%%%%%%%%%
%%%%%%%%%%%%%%%%%%%%%%%%%%%%%%%%%%%%%%%%%%%%
%%%%%%%%%%%%%%%%%%%%%%%%%%%%%%%%%%%%%%%%%%%%
%%%%%%%%%%%%%%%%%%%%%%%%%%%%%%%%%%%%%%%%%%%%
%%%%%%%%%%%%%%%%%%%%%%%%%%%%%%%%%%%%%%%%%%%%
%%%%%%%%%%%%%%%%%%%%%%%%%%%%%%%%%%%%%%%%%%%%
%%%%%%%%%%%%%%%%%%%%%%%%%%%%%%%%%%%%%%%%%%%%
%%%%%%%%%%%%%%%%%%%%%%%%%%%%%%%%%%%%%%%%%%%%
%%%%%%%%%%%%%%%%%%%%%%%%%%%%%%%%%%%%%%%%%%%%
%%%%%%%%%%%%%%%%%%%%%%%%%%%%%%%%%%%%%%%%%%%%
\section{LLM Usage}

Large Language models played a significant role in the development of this work. With the suggestion that an analog of the proof of Mazur's weak basis theorem may be useful for proving automatic continuity (which was suspected by C.~Rosendal but had not been worked out in any detail) GPT 5.5 largely produced the above proof that every order basis $(e_n)$ for a Banach lattice $X=[e_n]$ is automatically a Schauder basis. Regarding the complexity questions, although we were largely convinced that the problems we left open in \cite{ARTT} were within reach, we had already spent a significant amount of time coming up with the ideas in \cite{ARTT}, so it was natural to leave them for the future. To obtain the results in this paper, we used GPT 5.5  to provide the technical construction based on the ideas in  \cite{ARTT}, completing the theory in \cite{ARTT}. The fact that GPT 5.5  was able to do this is perhaps not so surprising, as there is a massive literature on both descriptive set theory and Banach lattices, but there are less than half a dozen papers in their intersection. However, the speed at which it was able to upgrade our construction to its rightful generality was impressive.

We also acknowledge that Claude Opus 4.8 and GPT 5.6 Sol were used to assist in the Lean verification. The statements in the corresponding code were carefully reviewed by the authors.

\section{Preliminaries}\label{sec:prelim}

All Banach lattices are real. Throughout, $\N=\{1,2,\ldots\}$; in particular, $\N^{<\N}$ denotes the set of finite strings of positive integers, together with the empty string $\tom$. We use the notation and conventions of \cites{ARTT,HOF}. If $z_m\in X_+$ is decreasing and has infimum $0$, we write $z_m\downarrow0$. Also, $\all^\infty n$ means ``for all but finitely many $n$''. A sequence $(x_n)$ in $X$ is said to $\sigma$-order converge to $x$, written
\[
x_n\conv{\sigma\sf o}{n}x,
\]
if there is a sequence $z_m\downarrow0$ in $X_+$ such that
\[
\all m\;\all^\infty n\quad |x_n-x|\leqslant z_m.
\]
It order converges to $x$, written $x_n\conv{\sf o}{n}x$, if the same assertion holds with a decreasing net $z_\mu\downarrow0$ in place of the sequence $(z_m)$. Uniform convergence, denoted $x_n\conv{\sf u}{n}x$, means that there is a $z\in X_+$ such that
\[
\all m\;\all^\infty n\quad |x_n-x|\leqslant \frac{z}{m}.
\]

If $C$ is one of these convergences, then a sequence $(e_n)$ in $X$ is called a $C$-basis for $X$ if every $x\in X$ has a unique scalar expansion
\[
x=^C\sum_{n=1}^{\infty}a_ne_n.
\]
In this case, the corresponding coordinate functionals are denoted by $e_n^\sharp$, and the partial sum maps by
\[
P_mx=\sum_{n=1}^m e_n^\sharp(x)e_n.
\]

We reserve the notation $e_n^*$ for continuous coordinate functionals.

\begin{rem}
The terminology follows Gumenchuk, Karlova and Popov \cite{gumenchuk}, except that their order bases are here called $\sigma$-order bases. See also \cites{taylor} for further information on bases in Banach lattices. %, or the classical Riesz-space usage of order basis for a band-generating set \cite{LZ}*{Definition~28.1}. Their Problem~1.3 is answered by Theorem~A, without the order-continuity assumption. Related bibasic and Haar-system questions appear in \cites{taylor,popov}.
\end{rem}

We shall use the following reduction from order to $\sigma$-order convergence. It is proved in \cite{ARTT}, but we recall the short proof for convenience.

\begin{prop}\label{prop:order-sigma}
\leanformalization{\leanstatement{TheoremA/Lean_Code/ThmA/OrderSigma.lean}{119}{orderConvergesTo_iff_sigmaOrderTendsto}}
Let $X$ be a separable Banach lattice. Then a sequence in $X$ order converges to $x\in X$ if and only if it $\sigma$-order converges to $x$.
\end{prop}

\begin{proof}
Only one implication needs proof. Suppose $x_n\conv{\sf o}{n}x$ and put
\[
A=\Mgd{z\in X_+}{\all^\infty n\; |x_n-x|\leqslant z}.
\]
Then $A$ is closed under finite meets and $\inf A=0$. Since $X$ is separable, we may choose a countable subset $D=\{d_1,d_2,\ldots\}$ that is dense in $A$ with its inherited norm topology, and set
\[
z_m=d_1\wedge\cdots\wedge d_m.
\]
Then $z_m\in A$ and $(z_m)$ is decreasing. If $0\leqslant h\leqslant z_m$ for every $m$, then $h\leqslant d_i$ for every $i$. By density of $D$ in $A$ and the fact that the order relation is norm closed, we have $h\leqslant z$ for every $z\in A$. Since $\inf A = 0$, this forces $h=0$. Thus, $z_m\downarrow0$ and the sequence $(z_m)$ witnesses the convergence $x_n\conv{\sigma\sf o}{n}x$.
\end{proof}

For a Banach lattice $X$, we define
\[
X_\downarrow=\Mgd{(x_n)\in X^\N}{x_1\geqslant x_2\geqslant\ldots\geqslant0}
\]
and
\[
X_{\downarrow0}=\Mgd{(x_n)\in X_\downarrow}{\inf_nx_n=0}.
\]
Similarly, we define
\[
X^\uparrow=\Mgd{(x_n)\in X^\N}{0\leqslant x_1\leqslant x_2\leqslant\ldots}
\]
and
\[
X^{\uparrow\infty}=\Mgd{(x_n)\in X^\uparrow}{(x_n)\text{ is not order bounded in }X}.
\]
Thus, $(x_n)\in X^{\uparrow\infty}$ if it is an increasing positive sequence and there is no $y\in X$ satisfying $x_n\leqslant y$ for all $n$.

\begin{prop}\label{prop:coanalytic}
\leanformalization{\leanstatement{ThmBC/Lower/Complete.lean}{250}{xDownarrowZeroZ_coanalytic} and
  \leanstatement{ThmBC/Upper/Upper.lean}{255}{yUpwardUnboundedY_coanalytic}}
For every separable Banach lattice $X$, both $X_{\downarrow0}$ and $X^{\uparrow\infty}$ are coanalytic subsets of $X^\N$.
\end{prop}

\begin{proof}
The sets $X_\downarrow$ and $X^\uparrow$ are closed in $X^\N$, since the order relation is norm closed. For $(x_n)\in X_\downarrow$, the assertion that $\inf_nx_n\ne0$ can be expressed as
\[
\ex y\in X\quad \big(y\geqslant0,\;\norm y>0,\; y\leqslant x_n\text{ for all }n\big).
\]
The relation inside the parentheses is Borel and so the complement of $X_{\downarrow0}$ is the union of the Borel set $X^\N\setminus X_\downarrow$ and an analytic set.

For the increasing case, note that if $0\leqslant x_1\leqslant x_2\leqslant\ldots$ has an upper bound $y\in X$, then $y\geqslant0$. Thus, the order bounded part of $X^\uparrow$ is
\[
\Mgd{(x_n)\in X^\uparrow}{\ex y\in X_+\;\all n\; x_n\leqslant y},
\]
which is the projection of a closed subset of $X^\N\times X$. Hence, it is analytic, and its complement in the closed set $X^\uparrow$ is coanalytic.
\end{proof}

We also recall the standard coding of well-founded trees. A tree on $\N$ is a downward closed subset $T\subseteq\N^{<\N}$ containing the empty string $\tom$. The space of all such trees, viewed as a closed subspace of $2^{\N^{<\N}}$, is denoted by $\Tr$. The subspace of well-founded trees, i.e., those having no infinite branch, is denoted by $\WF$. It is classical that $\WF$ is $\mathbf{\Pi}^1_1$-complete \cite{Kechris}*{Section 27}.

%%%%%%%%%%%%%%%%%%%%%%%%%%%%%%%%%%%%%%%%%%%%
%%%%%%%%%%%%%%%%%%%%%%%%%%%%%%%%%%%%%%%%%%%%
%%%%%%%%%%%%%%%%%%%%%%%%%%%%%%%%%%%%%%%%%%%%
%%%%%%%%%%%%%%%%%%%%%%%%%%%%%%%%%%%%%%%%%%%%
%%%%%%%%%%%%%%%%%%%%%%%%%%%%%%%%%%%%%%%%%%%%
%%%%%%%%%%%%%%%%%%%%%%%%%%%%%%%%%%%%%%%%%%%%
%%%%%%%%%%%%%%%%%%%%%%%%%%%%%%%%%%%%%%%%%%%%

\section{Order bases are Schauder bases}\label{sec:bases}

Throughout this section, $X$ is a Banach lattice and $(e_n)$ is a $\sigma$-order basis for $X=[e_n]$. The coordinate maps $e_n^\sharp$ and the projections $P_m$ are defined by
\[
P_mx=\sum_{n=1}^me_n^\sharp(x)e_n,
\qquad
P_mx\conv{\sigma\sf o}{m}x.
\]
We record the following elementary facts.
\begin{lemme}\label{lem:order-facts}\leanformalization{
  \leanstatement{TheoremA/Lean_Code/ThmA/SigmaOrder.lean}{32}{SigmaOrderTendsto.add},
  \leanstatement{TheoremA/Lean_Code/ThmA/SigmaOrder.lean}{49}{SigmaOrderTendsto.smul},
  \leanstatement{TheoremA/Lean_Code/ThmA/SigmaOrder.lean}{62}{SigmaOrderTendsto.orderBounded},
  \leanstatement{TheoremA/Lean_Code/ThmA/SigmaOrder.lean}{82}{SigmaOrderTendsto.abs_le},
  \leanstatement{TheoremA/Lean_Code/ThmA/SigmaOrder.lean}{99}{abs_le_of_tendsto}, and
  \leanstatement{TheoremA/Lean_Code/ThmA/SigmaOrder.lean}{107}{downToZero_tsum_tail}}
Let $X$ be a Banach lattice.
\begin{enumerate}
\item $\sigma$-order convergence is compatible with addition and scalar multiplication.
\item If $x_n\conv{\sigma\sf o}{n}x$, then the set $\{x_n\del n\in\N\}$ is order bounded.
\item If $x_n\conv{\sigma\sf o}{n}x$ and $u\in X_+$ satisfies $|x_n|\leqslant u$ for all $n$, then $|x|\leqslant u$.
\item If $z_i\to z$ in norm and $|z_i|\leqslant u$ for all $i$, where $u\in X_+$, then $|z|\leqslant u$.
\item If $\sum_{j=1}^{\infty}u_j$ is a norm-convergent series in $X_+$ and $U_m=\sum_{j=m}^{\infty}u_j$, then $U_m\downarrow0$.
\end{enumerate}
\end{lemme}

From the above lemma it follows that the functionals $e_n^\sharp$ and the maps $P_m$ are linear and that the vectors $(e_n)$ are linearly independent. For $x\in X$, we define
\begin{equation}\label{eq:rho}
\rho(x):=\inf\Mgd{\norm u}{u\in X_+\text{ and } |P_mx|\leqslant u\text{ for all }m\in\N}.
\end{equation}
Note that the collection of admissible $u$ in \eqref{eq:rho} is non-empty by Lemma~\ref{lem:order-facts}(2).

\begin{lemme}\label{lem:rho-norm}
\leanformalization{
  \leanstatement{TheoremA/Lean_Code/ThmA/AuxNorm.lean}{51}{norm_le_rho},
  \leanstatement{TheoremA/Lean_Code/ThmA/AuxNorm.lean}{78}{rho_eq_zero_iff},
  \leanstatement{TheoremA/Lean_Code/ThmA/AuxNorm.lean}{106}{rho_smul},\\
  \leanstatement{TheoremA/Lean_Code/ThmA/AuxNorm.lean}{121}{rho_add_le}, and
  \leanstatement{TheoremA/Lean_Code/ThmA/AuxNorm.lean}{194}{instNormedAddCommGroupXrho}}
The functional $\rho$ is a norm on $X$ and
\begin{equation}\label{eq:norm-le-rho}
\norm x\leqslant\rho(x)\qquad(x\in X).
\end{equation}
\end{lemme}

\begin{proof}
Homogeneity is immediate. If $u,v\in X_+$ dominate, respectively, all $|P_mx|$ and all $|P_my|$, then, by linearity, $u+v$ dominates all $|P_m(x+y)|$. Thus, $\rho(x+y)\leqslant\rho(x)+\rho(y)$.

If $u$ is admissible for $x$, then $P_mx\conv{\sigma\sf o}{m}x$ and Lemma~\ref{lem:order-facts}(3) yields $|x|\leqslant u$. Hence, $\norm x\leqslant\norm u$. Taking the infimum over admissible $u$ gives \eqref{eq:norm-le-rho}; in particular, $\rho(x)=0$ implies $x=0$.
\end{proof}

\begin{prop}\label{prop:rho-complete}
\leanformalization{\leanstatement{TheoremA/Lean_Code/ThmA/Complete.lean}{249}{instCompleteSpaceXrho}}
The normed space $(X,\rho)$ is complete.
\end{prop}

\begin{proof}
Let $(x_j)$ be a $\rho$-Cauchy sequence. It suffices to prove that a subsequence converges with respect to $\rho$, since a Cauchy sequence with a convergent subsequence converges to the same limit. Choose such a subsequence, relabeled as $(x_r)$, so that
\[
\sum_{r=1}^{\infty}\rho(x_{r+1}-x_r)<\infty,
\]
and put $\Delta_r=x_{r+1}-x_r$. For each $r$, choose $u_r\in X_+$ satisfying
\begin{equation}\label{eq:ur-bound}
|P_m\Delta_r|\leqslant u_r\quad(m\in\N)
\end{equation}
and
\[
\norm{u_r}\leqslant\rho(\Delta_r)+2^{-r}.
\]
Then $\sum_r\norm{u_r}<\infty$. By Lemma~\ref{lem:rho-norm}, we also have $\sum_r\norm{\Delta_r}<\infty$. For $R\in\N$, set
\[
U_R=\sum_{r=R}^{\infty}u_r,
\qquad
T_R=\sum_{r=R}^{\infty}|\Delta_r|.
\]
By Lemma~\ref{lem:order-facts}(5), we have
\begin{equation}\label{eq:tails-down}
U_R\downarrow0,
\qquad
T_R\downarrow0.
\end{equation}
Moreover, for $S>R$ we may write
\[
|x_S-x_R|\leqslant\sum_{r=R}^{S-1}|\Delta_r|\leqslant T_R,
\]
so $(x_R)$ is Cauchy for the norm $\norm{\cdot}$. Let $x=\lim_Rx_R$ be the limit with respect to this norm. Passing to the norm limit in the above inequality gives
\begin{equation}\label{eq:x-xR}
|x-x_R|\leqslant T_R.
\end{equation}
Fix $m$. From \eqref{eq:ur-bound}, we have
\[
|P_mx_S-P_mx_R|
\leqslant\sum_{r=R}^{S-1}|P_m\Delta_r|
\leqslant U_R
\qquad(S>R).
\]
Thus, $(P_mx_R)_R$ is Cauchy for the norm $\norm{\cdot}$. Let
\[
y_m=\lim_{R\to\infty}P_mx_R.
\]
Again, passing to the norm limit gives
\begin{equation}\label{eq:ym-pmxR}
|y_m-P_mx_R|\leqslant U_R
\qquad(m,R\in\N).
\end{equation}
We claim that $y_m\conv{\sigma\sf o}{m}x$. We begin by fixing $R$. Since $P_mx_R\conv{\sigma\sf o}{m}x_R$, we may choose a regulator $w^R_j\downarrow0$ so that, for every $j$, we have the inequality
\begin{equation}\label{eq:wR}
|P_mx_R-x_R|\leqslant w_j^R
\end{equation}
for all sufficiently large $m$. Combining \eqref{eq:x-xR}, \eqref{eq:ym-pmxR} and \eqref{eq:wR} we find that, for every fixed $R$ and $j$, we have
\begin{equation}\label{eq:y-x-eventual}
|y_m-x|\leqslant U_R+T_R+w_j^R
\end{equation}
for all sufficiently large $m$.

For $j\in\N$, put
\begin{equation}\label{eq:bj}
b_j=\bigwedge_{R=1}^j\big(U_R+T_R+w_j^R\big).
\end{equation}
Clearly, $(b_j)$ is decreasing. For fixed $j$, the finitely many estimates \eqref{eq:y-x-eventual} with $R=1,\ldots,j$ hold simultaneously for all sufficiently large $m$, and therefore the inequality $|y_m-x|\leqslant b_j$ holds eventually.
It remains to see that $b_j\downarrow0$. Suppose, to the contrary, that $0\leqslant h\leqslant b_j$ for all $j$. Fix $R$. For every $j\geqslant R$,  the definition \eqref{eq:bj} gives
\[
h\leqslant U_R+T_R+w_j^R.
\]
Letting $j\to\infty$ and using that $w_j^R\downarrow0$, we get $(h-(U_R+T_R))^+=0$, and so $h\leqslant U_R+T_R$. This holds for all $R$. Since $U_R,T_R\downarrow0$, we also have $U_R+T_R\downarrow0$, from which it follows that $h=0$. Thus, $b_j\downarrow0$ and $y_m\conv{\sigma\sf o}{m}x$, as claimed.

Now, since the vectors $P_mx_R$ belong to the finite-dimensional space $E_m=\spa\{e_1,\ldots,e_m\}$, we deduce that $y_m\in E_m$ as well. Consequently, there are scalars $(a_n)$ such that
\[
y_m=\sum_{n=1}^ma_ne_n.
\]
Since $y_m\conv{\sigma\sf o}{m}x$, uniqueness of the $\sigma$-order expansion gives $a_n=e_n^\sharp(x)$, i.e.,
\begin{equation}\label{eq:ym-pmx}
y_m=P_mx.
\end{equation}
Combining \eqref{eq:ym-pmxR} and \eqref{eq:ym-pmx}, we obtain
\[
|P_m(x-x_R)|\leqslant U_R
\qquad(m,R\in\N).
\]
Thus, $U_R$ is admissible in the definition of $\rho(x-x_R)$ and
\[
\rho(x-x_R)\leqslant\norm{U_R}\longrightarrow0.
\]
This proves that our chosen subsequence converges with respect to $\rho$, and hence the original $\rho$-Cauchy sequence converges in the $\rho$ norm.
\end{proof}

We can now proceed with the proof of our main result in this section.
%\begin{thm}\label{thm:bases-main}
%Let $(e_n)$ be a $\sigma$-order basis for a Banach lattice $X=[e_n]$. Then $(e_n)$ is a Schauder basis for $X$. Moreover, there is a constant $C<\infty$ so that, for every $N$ and all scalars $a_1,\ldots,a_N$, we have
%\begin{equation}\label{eq:maxineq}
%\BNORM{\bigvee_{m=1}^N\left|\sum_{k=1}^ma_ke_k\right|}
%\leqslant C\Norm{\sum_{k=1}^Na_ke_k}.
%\end{equation}
%\end{thm}

\begin{proof}[Proof of Theorem \ref{theoremA}]
By Proposition \ref{prop:order-sigma} it is enough to assume $(e_n)$ is a $\sigma$-order basis for a Banach lattice $X=[e_n]$. Now, by Proposition~\ref{prop:rho-complete}, $(X,\rho)$ is a Banach space. Lemma~\ref{lem:rho-norm} says that the identity map
\[
I:(X,\rho)\longrightarrow(X,\norm{\cdot})
\]
is continuous and bijective. Hence, the inverse mapping theorem yields a constant $C<\infty$ with
\begin{equation}\label{eq:rho-C}
\rho(x)\leqslant C\norm x
\qquad(x\in X).
\end{equation}
Suppose we have $|P_mx|\leqslant u$ for every $m\in \N$. In particular, $\norm{P_mx}\leqslant\norm u$. Taking the infimum over all such $u$ gives
\[
\norm{P_mx}\leqslant\rho(x)\leqslant C\norm x.
\]
Thus, the maps $P_m$ are uniformly bounded in the original norm.

Since $X=[e_n]$, the union of the finite-dimensional spaces $\spa\{e_1,\ldots,e_N\}$ is norm dense in $X$. If $y\in\spa\{e_1,\ldots,e_N\}$, then the finite coefficient sequence for $y$, followed by zeros, is a $\sigma$-order expansion; by uniqueness, $P_my=y$ for all $m\geqslant N$. Hence, for $m\geqslant N$ we have
\[
\norm{P_mx-x}\leqslant\norm{P_m(x-y)}+\norm{y-x}\leqslant(C+1)\norm{x-y}.
\]
Approximating $x$ by such a $y$, we conclude that $P_mx\to x$ in norm.

Next, we note that the norm expansion is unique. Indeed, if $\sum_{n=1}^M b_ne_n\to x$ in norm, then for fixed $m$ and all $M\geqslant m$, we have
\[
P_m\left(\sum_{n=1}^M b_ne_n\right)=\sum_{n=1}^mb_ne_n.
\]
 Using the continuity of the $P_m$ to pass to the norm limit gives $P_mx=\sum_{n=1}^mb_ne_n$, so $b_n=e_n^\sharp(x)$ by linear independence. Thus, $(e_n)$ is Schauder.

Finally, we verify the maximal inequality. Take $x=\sum_{n=1}^Na_ne_n$. Then the element
\[
v=\bigvee_{m=1}^N\left|\sum_{k=1}^ma_ke_k\right|
\]
is the least positive order bound of the family $(P_mx)_{m\geqslant1}$, since $P_mx=x$ for $m\geqslant N$. Thus, $v$ is admissible in \eqref{eq:rho}, and every admissible $u$ satisfies $v\leqslant u$. By monotonicity of the lattice norm, we have $\rho(x)=\norm v$, and thus \eqref{eq:rho-C} implies \eqref{eq:maxineq}.
\end{proof}

%\begin{cor}\label{cor:order-basis-zfc}
%If $(e_n)$ is an order basis or a $\sigma$-order basis for a Banach lattice $X=[e_n]$, then $(e_n)$ is a Schauder basis. In fact, it is a uniform basis.
%\end{cor}

%\begin{proof}
%The $\sigma$-order case is Theorem~\ref{thm:bases-main}. If $(e_n)$ is an order basis, then $X=[e_n]$ is separable, so order and $\sigma$-order convergence agree by Proposition~\ref{prop:order-sigma}. Thus $(e_n)$ is a $\sigma$-order basis. The maximal inequality \eqref{eq:maxineq}, together with \cite{taylor}*{Theorem 2.1}, gives uniform convergence of the partial sums, hence a uniform basis.
%\end{proof}

%%%%%%%%%%%%%%%%%%%%%%%%%%%%%%%%%%%%%%%%%%%
%%%%%%%%%%%%%%%%%%%%%%%%%%%%%%%%%%%%%%%%%%%
%%%%%%%%%%%%%%%%%%%%%%%%%%%%%%%%%%%%%%%%%%%
%%%%%%%%%%%%%%%%%%%%%%%%%%%%%%%%%%%%%%%%%%%
%%%%%%%%%%%%%%%%%%%%%%%%%%%%%%%%%%%%%%%%%%%
%%%%%%%%%%%%%%%%%%%%%%%%%%%%%%%%%%%%%%%%%%%
%%%%%%%%%%%%%%%%%%%%%%%%%%%%%%%%%%%%%%%%%%%
%%%%%%%%%%%%%%%%%%%%%%%%%%%%%%%%%%%%%%%%%%%
%%%%%%%%%%%%%%%%%%%%%%%%%%%%%%%%%%%%%%%%%%%

\section{Generalities on the weighted tree construction}\label{sec:tree-template}

We now prepare for the two descriptive-set-theoretic constructions. If $s,t\in\N^{<\N}$, then $s\concat t$ denotes concatenation, and if $n\in\N$ we use $s\concat n$ as an abbreviation for $s\concat(n)$. For a finite string $s=(s_1,\ldots,s_k)$, put $|s|=k$ and, when $s\ne\tom$, let $\last(s)=s_k$.

Fix once and for all the constants
\[
\gamma=\frac1{10},
\qquad
\lambda_n=2^{-n},\quad(n\in\N).
\]
Define weights on $\N^{<\N}$ by
\[
\omega(\tom)=1
\]
and for $s=(s_1,\ldots,s_k)$,
\[
\omega(s)=\gamma^k\prod_{i=1}^{k-1}\lambda_{s_i},\quad(k\geq 1).
\]
With this definition, we have $\omega(n)=\gamma$ and, if $s\ne\tom$, then
\begin{equation}\label{eq:weights-child}
\omega(s\concat n)=\gamma\lambda_{\last(s)}\omega(s),
\end{equation}
which is independent of  $n$.

Let $E$ be a separable Banach lattice. We denote by $V_E$ the Banach lattice of all indexed sequences
\[
x=(x_s)_{s\in\N^{<\N}}
\]
with $x_s\in E$ and
\[
\norm{x}_{V_E}=\sup_{s\in\N^{<\N}}\omega(s)\norm{x_s}_E<\infty,
\]
with the order and lattice operations on $V_E$ defined coordinatewise.

\begin{defi}\label{def:finite-rank}
An element $x\in V_E$ is called finite-rank if there is a finite downward closed tree $F\subseteq\N^{<\N}$ containing $\tom$ such that, whenever $s\in F$ and $s\concat n\notin F$, the whole subtree above $s\concat n$ is constant with value $x_s$, i.e.,
\[
x_{s\concat n\concat t}=x_s
\qquad(t\in\N^{<\N}).
\]
Let $\ku S_E$ denote the set of finite-rank elements and put
\[
Z_E=\ov{\ku S_E}^{\norm{\cdot}_{V_E}}.
\]
\end{defi}

\begin{lemme}\label{lem:S-template}
\leanformalization{\leanstatement{ThmBC/Tree/Dense.lean}{293}{finiteRank_template}}
The set $\ku S_E$ is a vector sublattice of $V_E$. Consequently, $Z_E$ is a closed Banach sublattice of $V_E$. Moreover, $Z_E$ is separable.
\end{lemme}

\begin{proof}
If $F$ witnesses that $x$ is finite-rank and $F\subseteq H$ is finite and downward closed, then $H$ also witnesses that $x$ is finite-rank. Thus, for $x,y\in\ku S_E$, we may choose a common finite downward closed witness $H$. Outside $H$, both elements are constant on the omitted subtree, and therefore so are $x+y$, $x-y$, $x\vee y$ and $x\wedge y$. Hence, $\ku S_E$ is a vector sublattice, and its norm closure is a closed Banach sublattice.

For separability, let $D\subseteq E$ be countable and norm dense. For a fixed finite downward closed witness, a finite-rank element is determined by finitely many values in $E$. There are only countably many such finite trees, and elements whose defining values lie in $D$ form a countable dense subset of $\ku S_E$, hence of $Z_E=\ov{\ku S_E}$.
\end{proof}

\begin{lemme}\label{lem:child-convergence-template}
\leanformalization{\leanstatement{ThmBC/Tree/Dense.lean}{374}{child_convergence_template}}
For every $x\in Z_E$ and every $s\in\N^{<\N}$, we have
\[
x_{s\concat n}\longrightarrow x_s
\qquad\text{in }E.
\]
\end{lemme}

\begin{proof}
For finite-rank elements this is immediate from the definition. Let $x\in Z_E$, fix $s\in\N^{<\N}$ and let $\eps>0$. Choose $u\in\ku S_E$ so close to $x$ in $V_E$ that
\[
\norm{x_s-u_s}_E<\eps/3
\]
and
\[
\norm{x_{s\concat n}-u_{s\concat n}}_E<\eps/3
\qquad(n\in\N).
\]
The latter uniform estimate is possible because the weights of the children $s\concat n$ are independent of $n$ by \eqref{eq:weights-child}, and also at the root, where $\omega(n)=\gamma$. Since $u_{s\concat n}=u_s$ for all sufficiently large $n$, we get $\norm{x_{s\concat n}-x_s}_E<\eps$ eventually.
\end{proof}

\begin{lemme}\label{lem:eventually-constant-template}
\leanformalization{\leanstatement{ThmBC/Tree/Eventual.lean}{246}{eventually_constant_template}}
Let $u=(u_s)_{s\in\N^{<\N}}\in V_E$ be such that
\[
M=\sup_s\norm{u_s}_E<\infty.
\]
Assume that, for every $s\in\N^{<\N}$, there is $N_s$ so that
\[
u_{s\concat n}=u_s
\qquad(n\geqslant N_s).
\]
Then $u\in Z_E$.
\end{lemme}

\begin{proof}
If $M=0$, there is nothing to prove. For a non-empty node $r$, put
\[
B(r)=\sup\Mgd{\omega(t)}{t\text{ is a proper extension of }r}.
\]
A direct computation gives
\begin{equation}\label{eq:B-r}
B(r)=\gamma\lambda_{\last(r)}\omega(r).
\end{equation}
Indeed, the largest weight strictly below $r$ occurs at an immediate child of $r$, and every further descent multiplies by a factor at most $\gamma\lambda_1<1$.

For every $\eta>0$, the set $\{r\ne\tom\del B(r)\geqslant\eta\}$ is finite. Indeed, if $|r|=k$, then $B(r)\leqslant\gamma(\gamma/2)^k$, which bounds $k$, and at each fixed level only finitely many products $\prod_i2^{-r_i}$ can exceed a prescribed positive number. In particular,
\[
\sup\Mgd{\omega(t)}{|t|>L}\longrightarrow0
\qquad(L\to\infty),
\]
and, for each fixed $s$, $B(s\concat n)\to0$ as $n\to\infty$.

Fix $\eps>0$. Choose $L$ so large that
\[
2M\sup\Mgd{\omega(t)}{|t|>L}<\eps.
\]
We build a finite downward closed tree $F$. Start with $\tom\in F$. Recursively, and stopping after level $L$, whenever $s\in F$ and $|s|<L$, choose $K_s\geqslant N_s$ so large that
\[
2M B(s\concat n)<\eps
\qquad(n\geqslant K_s),
\]
and put $s\concat n$ into $F$ for $1\leqslant n<K_s$. No descendants are added below nodes of depth $L$.

Define a finite-rank element $v$ by setting $v_s=u_s$ for $s\in F$ and, whenever $s\in F$ but $s\concat n\notin F$, declare the whole subtree above $s\concat n$ to be constantly $u_s$. We claim that $\norm{u-v}_{V_E}\leqslant\eps$. Note that if $t\in F$, there is no error. Otherwise, $t$ lies in the omitted subtree above a first node $s\concat n\notin F$ with $s\in F$. If $|s|=L$, then $|t|>L$, and so
\[
\omega(t)\norm{u_t-v_t}_E\leqslant2M\omega(t)<\eps.
\]
If $|s|<L$, then $n\geqslant K_s\geqslant N_s$, so $u_{s\concat n}=u_s=v_{s\concat n}$, and at the node $s\concat n$ itself there is no error. Finally, at every proper descendant $t$ of $s\concat n$, we have $\omega(t)\leqslant B(s\concat n)$, and hence
\[
\omega(t)\norm{u_t-v_t}_E\leqslant2M B(s\concat n)<\eps.
\]
Thus, $u$ is in the closure of $\ku S_E$.
\end{proof}

We shall use the following elementary operation on trees. For $T\in\Tr$, define its fattening $\widehat T$ by
\[
\widehat T=\Mgd{s=(s_1,\ldots,s_k)\in\N^{<\N}}{\ex a=(a_1,\ldots,a_k)\in T\text{ with }a_i\leqslant s_i\text{ for all }i}.
\]
For $k=0$, this just says that $\tom\in\widehat T$. It is immediate from the downward closedness of $T$ that $\widehat T$ is again a tree.

\begin{lemme}\label{lem:fattening}
\leanformalization{\leanstatement{ThmBC/Tree/Tree.lean}{347}{fattening_template}}
The map $T\mapsto\widehat T$ is continuous from $\Tr$ to $\Tr$ and preserves well-foundedness, in the sense that
\[
T\in\WF\quad\equi\quad\widehat T\in\WF.
\]
Moreover, if $b=(b_i)_{i=1}^\infty$ is an infinite branch through $T$, then
\[
L_b=\Mgd{(n_1,\ldots,n_k)}{b_i\leqslant n_i\text{ for }1\leqslant i\leqslant k}
\]
is a subtree of $\widehat T$ and, whenever $s\in L_b$, one has $s\concat n\in L_b$ for all sufficiently large $n$.
\end{lemme}

\begin{proof}
For fixed $s=(s_1,\ldots,s_k)$, the membership $s\in\widehat T$ depends only on the finitely many nodes in $\prod_{i=1}^k\{1,\ldots,s_i\}$, so the map is continuous. If $T$ has an infinite branch, the same branch witnesses that $\widehat T$ is ill-founded. Conversely, suppose that $\widehat T$ has an infinite branch $s=(s_i)_{i=1}^{\infty}$. For each $k$, choose $a^{(k)}\in T$ of length $k$ with $a_i^{(k)}\leqslant s_i$ for $1\leqslant i\leqslant k$. Then
\[
B=\Mgd{a\in T}{a_i\leqslant s_i\text{ for all }1\leqslant i\leqslant |a|}
\]
is a finitely branching tree with nodes at every level. By K\"onig's lemma, $B$, and hence $T$, has an infinite branch.

The final assertion follows directly from the definition of $L_b$, since if we have $s=(n_1,\ldots,n_k)\in L_b$ then we also have $s\concat n\in L_b$ for all $n\geqslant b_{k+1}$.
\end{proof}

\section{A Banach lattice \texorpdfstring{$X$}{X} for which
\texorpdfstring{$X_{\downarrow0}$}{X downarrow 0} is complete coanalytic}\label{sec:lower-complete}

In this section, we take
\[
E=L_1[0,1],
\qquad
e=\mathbf{1}_{[0,1]}
\]
and let
\[
X=Z_E
\]
be the separable Banach lattice constructed in the previous section. We  define a finite-rank element $d_m\in X_+$ by
\[
(d_m)_\tom=e
\]
and, for $s=(n)\concat t$,
\[
(d_m)_s=
\begin{cases}
0, & n<m,\\
e, & n\geqslant m.
\end{cases}
\]
The following lemma records the basic properties of $(d_m)$.
\begin{lemme}\label{lem:dead-sequence}
\leanformalization{\leanstatement{ThmBC/Lower/Continuity.lean}{56}{dead_sequence_template}}
For every $m$, $d_m\in X_+$. Moreover,
\[
d_1\geqslant d_2\geqslant\cdots\geqslant0
\]
and $\inf_md_m=0$ in $X$.
\end{lemme}

\begin{proof}
 The monotonicity and positivity assertions are clear. Suppose next that $0\leqslant x\in X$ and $x\leqslant d_m$ for all $m$. Fix a first-level coordinate $n$. If $m>n$, then the entire $n$th first-level subtree of $d_m$ is zero, and hence the same subtree of $x$ is zero. Thus, $x_n=0$ for every first-level node $n$. By Lemma~\ref{lem:child-convergence-template}, $x_n\to x_\tom$ in $E$, so $x_\tom=0$. Since all other coordinates have already been shown to be zero, we may conclude that $x=0$.
\end{proof}

Let $R\in\Tr$. For each $m$, define an element $z_m^R\in V_E$ as follows: If $s\in R$, set
\[
(z_m^R)_s=e.
\]
If $s\notin R$, let $q\concat n$ be the first initial segment of $s$ outside $R$, with $q\in R$, and write
\[
s=q\concat n\concat r.
\]
Then, set
\[
(z_m^R)_s=(d_m)_r.
\]
In other words, the indexed sequence defining this element stays constantly $e$ while it remains inside $R$, and at the first exit, a copy of the decreasing sequence $(d_m)$ is pasted into the exiting subtree. The basic properties of the elements $z_m^R$ are described in the following lemma.

\begin{lemme}\label{lem:z-lower-in-X}
\leanformalization{\leanstatement{ThmBC/Lower/Continuity.lean}{70}{z_lower_in_X_template}}
For every $R\in\Tr$ and every $m$, we have $z_m^R\in X_+$. Moreover,
\[
z_1^R\geqslant z_2^R\geqslant\cdots\geqslant0.
\]
\end{lemme}

\begin{proof}
The element $z_m^R$ is bounded by the constant indexed sequence with value $e$. If $s\in R$, then every child $s\concat n$ has value $e$, either because it still lies in $R$ or because the first exit occurs at that child and $(d_m)_\tom=e$. If $s\notin R$, then, above $s$, the element is governed by the finite-rank element $d_m$ and is therefore eventually constant along children. Lemma~\ref{lem:eventually-constant-template} gives $z_m^R\in X_+$. The monotonicity in $m$ follows from Lemma~\ref{lem:dead-sequence}.
\end{proof}

\begin{lemme}\label{lem:z-lower-continuous}
\leanformalization{
  \leanstatement{ThmBC/Lower/Continuity.lean}{363}{z_lower_continuous_template} and\\
  \leanstatement{ThmBC/Lower/Continuity.lean}{396}{z_lower_sequence_continuous_template}}
For every fixed $m$, the map
\[
\Tr\longrightarrow X,
\qquad
R\longmapsto z_m^R,
\]
is continuous. Consequently, the map $R\mapsto (z_m^R)_m$ is continuous from $\Tr$ to $X^\N$.
\end{lemme}

\begin{proof}
Let $B(r)$ be as in \eqref{eq:B-r}. For every $\eps>0$, the set
\[
A_\eps=\Mgd{r\ne\tom}{B(r)\geqslant\eps}
\]
is finite. Suppose $R$ and $R'$ agree on $A_\eps$. If $z_m^R$ and $z_m^{R'}$ differ at a node $t$, let $r$ be the first initial segment at which $R$ and $R'$ differ. The root cannot be such an $r$. Also, the two elements agree at $r$ itself, since the value at a first exit is $(d_m)_\tom=e$, which is the value inside the tree. Thus, any difference occurs strictly below $r$. Since $r\notin A_\eps$ and all values are either $0$ or $e$, with $\norm e=1$, we have
\[
\omega(t)\norm{(z_m^R)_t-(z_m^{R'})_t}_E\leqslant B(r)<\eps.
\]
Taking the supremum over $t$ proves continuity.
\end{proof}

Define
\[
\Theta:\Tr\longrightarrow X^\N,
\qquad
\Theta(T)=(z_m^{\widehat T})_{m=1}^{\infty}.
\]
By Lemmas~\ref{lem:fattening},~\ref{lem:z-lower-in-X} and~\ref{lem:z-lower-continuous}, the map $\Theta$ is continuous and $\Theta(T)\in X_\downarrow$ for every $T\in\Tr$. %For such a positive decreasing sequence, belonging to $X_{\downarrow0}$ is the same as having no non-zero positive lower bound: any lower bound $y$ gives the positive lower bound $y^+$.

\begin{lemme}\label{lem:lower-IF}
\leanformalization{\leanstatement{ThmBC/Lower/Reduction.lean}{157}{lower_IF_template}}
If $T\notin\WF$, then $\Theta(T)\notin X_{\downarrow0}$.
\end{lemme}

\begin{proof}
Let $a=(a_1,a_2,\ldots)$ be an infinite branch through $T$ and put $R=\widehat T$. Define $y^a\in V_E$ by
\[
(y^a)_s=
\begin{cases}
e, & s=(s_1,\ldots,s_k)\text{ and }a_i\leqslant s_i\text{ for all }1\leqslant i\leqslant k,\\
0, & \text{otherwise}.
\end{cases}
\]
At the root this value is $e$, so $y^a\ne0$. The element is bounded and eventually constant along children. Indeed, if $(y^a)_s=e$, then all sufficiently large children still dominate the branch, while if $(y^a)_s=0$, every descendant remains outside the domination cone. Hence, $y^a\in X_+$ by Lemma~\ref{lem:eventually-constant-template}.

Whenever $(y^a)_s=e$, the node $s$ dominates the corresponding initial segment of the branch $a$, and so $s\in\widehat T=R$. Therefore, $(z_m^R)_s=e$ for all $m$. Thus,
\[
0<y^a\leqslant z_m^R
\qquad(m\in\N),
\]
so the decreasing sequence $\Theta(T)$ has a non-zero positive lower bound.
\end{proof}

\begin{lemme}\label{lem:dead-subtree}
\leanformalization{\leanstatement{ThmBC/Lower/Reduction.lean}{209}{dead_subtree_template}}
Let $R\in\Tr$, suppose that $q\in R$ but $q\concat n\notin R$ and let $0\leqslant x\in X$ satisfy
\[
x\leqslant z_m^R\qquad(m\in\N).
\]
Then the entire subtree of $x$ above $q\concat n$ is zero.
\end{lemme}

\begin{proof}
On the subtree above $q\concat n$, the element $z_m^R$ is a copy of $d_m$. Hence, for every $r\in\N^{<\N}$,
\[
0\leqslant x_{q\concat n\concat r}\leqslant (d_m)_r
\qquad(m\in\N).
\]
Fix a first-level descendant $j$ inside this copied subtree. If $m>j$, then the whole $j$th first-level subtree of $d_m$ is zero, and so
\[
x_{q\concat n\concat j\concat t}=0
\qquad(t\in\N^{<\N}).
\]
This holds for every $j$. By Lemma~\ref{lem:child-convergence-template}, $x_{q\concat n\concat j}\to x_{q\concat n}$ in $E$, and hence $x_{q\concat n}=0$. Together with the already established vanishing below all children, this proves the assertion.
\end{proof}

\begin{lemme}\label{lem:lower-WF}
\leanformalization{\leanstatement{ThmBC/Lower/Reduction.lean}{427}{lower_WF_template}}
If $T\in\WF$, then $\Theta(T)\in X_{\downarrow0}$.
\end{lemme}

\begin{proof}
Let $R=\widehat T$. By Lemma~\ref{lem:fattening}, $R$ is well-founded. Suppose $0\leqslant x\in X$ and
\[
x\leqslant z_m^R\qquad(m\in\N).
\]
We prove by well-founded induction on $R$ that the whole subtree of $x$ above each node $q\in R$ is zero. Assume the conclusion has been proved for every immediate successor $q\concat n\in R$. Fix $n$. If $q\concat n\in R$, the induction hypothesis implies vanishing above that child. If $q\concat n\notin R$, Lemma~\ref{lem:dead-subtree} gives the same conclusion. Thus, $x_{q\concat n}=0$ for all $n$, and Lemma~\ref{lem:child-convergence-template} gives $x_q=0$. Hence, the whole subtree above $q$ is zero. Applying the conclusion at $q=\tom$ gives $x=0$.
\end{proof}

\begin{proof}[Proof of Theorem \ref{theoremB}]
For the separable Banach lattice $X=Z_{L_1[0,1]}$ constructed above, the map $\Theta:\Tr\to X^\N$ is continuous and satisfies
\[
T\in\WF
\quad\equi\quad
\Theta(T)\in X_{\downarrow0}
\]
by Lemmas~\ref{lem:lower-IF} and~\ref{lem:lower-WF}. Thus, $\WF\leq_{\rm c} X_{\downarrow0}$. Since $\WF$ is $\mathbf{\Pi}^1_1$-complete and $X_{\downarrow0}$ is coanalytic by Proposition~\ref{prop:coanalytic}, the result follows.
\end{proof}

\section{A Banach lattice \texorpdfstring{$Y$}{Y} for which
\texorpdfstring{$Y^{\uparrow\infty}$}{Y upward infinity} is complete coanalytic}\label{sec:upper-complete}

We now take $E=\R$ in the above template and write $Z_\R$ for the resulting separable Banach lattice of scalar indexed sequences. Let
\[
Y=\Mgd{y\in Z_\R}{y_\tom=0}.
\]
Since root evaluation is a continuous lattice homomorphism, $Y$ is a closed Banach sublattice of $Z_\R$, and hence is separable. %by Lemma~\ref{lem:S-template}.

For each $m\in\N$, define $a_m\in Y_+$ by
\[
(a_m)_\tom=0,
\]
and, for $j\in\N$ and $t\in\N^{<\N}$,
\[
(a_m)_{j\concat t}=
\begin{cases}
1, & j\leqslant m,\\
0, & j>m.
\end{cases}
\]
In other words, $a_m$ has root value $0$, its first $m$ first-level subtrees are constantly $1$, and all remaining first-level subtrees are constantly $0$. Note that each $a_m$ is finite-rank, and
\[
0\leqslant a_1\leqslant a_2\leqslant\cdots.
\]

\begin{lemme}\label{lem:local-unbounded}
\leanformalization{\leanstatement{ThmBC/Upper/Basic.lean}{176}{local_unbounded_template}}
The sequence $(a_m)$ is not order bounded in $Y$.
\end{lemme}

\begin{proof}
Suppose $y\in Y_+$ and $a_m\leqslant y$ for all $m$. For each fixed first-level node $j$, choose $m\geqslant j$. Then $1=(a_m)_j\leqslant y_j$. Hence, $y_j\geqslant1$ for every $j$. However, Lemma~\ref{lem:child-convergence-template} gives $y_j\to y_\tom=0$, a contradiction.
\end{proof}

Let $R\in\Tr$. For every $m$, define $u_m^R\in V_\R$ as follows. If $s\in R$, set
\[
(u_m^R)_s=0.
\]
If $s\notin R$, let $q\concat n$ be the first initial segment of $s$ outside $R$, with $q\in R$, and write $s=q\concat n\concat r$. Then set
\[
(u_m^R)_s=(a_m)_r.
\]
This definition is in analogy with the construction in the previous section, with the indexed sequence being $0$ while it remains inside $R$, and then having a copy of the bad sequence $(a_m)$ pasted in at the first exit.

\begin{lemme}\label{lem:u-upper-in-Y}
\leanformalization{\leanstatement{ThmBC/Upper/Continuity.lean}{168}{u_upper_in_Y_template}}
For every $R\in\Tr$ and every $m$, $u_m^R\in Y_+$. Moreover,
\[
0\leqslant u_1^R\leqslant u_2^R\leqslant\cdots.
\]
\end{lemme}

\begin{proof}
The element $u_m^R$ takes only the values $0$ and $1$. If $s\in R$, then all children have value $0$. Indeed, either the child is still in $R$, or the first exit occurs there and $(a_m)_\tom=0$. If $s\notin R$, the values on the subtree above $s$ are governed by the finite-rank element $a_m$ and hence are eventually constant along children. Lemma~\ref{lem:eventually-constant-template} gives $u_m^R\in Z_\R$, and since $(u_m^R)_\tom=0$, we have $u_m^R\in Y$. Positivity and monotonicity are clear.
\end{proof}

\begin{lemme}\label{lem:u-upper-continuous}
\leanformalization{\leanstatement{ThmBC/Upper/Continuity.lean}{178}{u_upper_continuous_template}}
For every fixed $m$, the map
\[
\Tr\longrightarrow Y,
\qquad
R\longmapsto u_m^R,
\]
is continuous. Consequently, $R\mapsto(u_m^R)_m$ is continuous from $\Tr$ to $Y^\N$.
\end{lemme}

\begin{proof}
The proof is identical to Lemma~\ref{lem:z-lower-continuous}. If two trees first differ at a node $r$, then the two indexed sequences agree at $r$ itself, because the value inside the tree is $0$ and the value at a first exit is $(a_m)_\tom=0$. Any difference occurs strictly below $r$, where the largest possible weight is $B(r)$. Since all values are in $\{0,1\}$ and the sets $\{r\ne\tom\del B(r)\geqslant\eps\}$ are finite, agreement on such a finite set forces the two elements to be within $\eps$ in norm.
\end{proof}

Define
\[
\Phi:\Tr\longrightarrow Y^\N,
\qquad
\Phi(T)=(u_m^{\widehat T})_{m=1}^{\infty}.
\]
By Lemmas~\ref{lem:fattening},~\ref{lem:u-upper-in-Y}, and~\ref{lem:u-upper-continuous}, the map $\Phi$ is continuous and $\Phi(T)\in Y^\uparrow$ for all $T\in\Tr$. %For such a positive increasing sequence, belonging to $Y^{\uparrow\infty}$ is the same as having no upper bound; any upper bound is automatically positive.

\begin{prop}\label{prop:upper-IF}
\leanformalization{\leanstatement{ThmBC/Upper/Reduction.lean}{24}{upper_IF_template}}
If $T\notin\WF$, then $\Phi(T)\notin Y^{\uparrow\infty}$.
\end{prop}

\begin{proof}
Let $b=(b_i)_{i=1}^{\infty}$ be an infinite branch through $T$ and put $R=\widehat T$. Define
\[
L_b=\Mgd{(n_1,\ldots,n_k)}{b_i\leqslant n_i\text{ for all }1\leqslant i\leqslant k}.
\]
Then $L_b\subseteq R$, and if $s\in L_b$, then $s\concat n\in L_b$ for all sufficiently large $n$.

Define a scalar indexed sequence $y^b$ by
\[
(y^b)_s=
\begin{cases}
0, & s\in L_b,\\
1, & s\notin L_b.
\end{cases}
\]
This indexed sequence is bounded and eventually constant along children, since if $s\in L_b$, sufficiently large children remain in $L_b$, while if $s\notin L_b$, every descendant is outside $L_b$. Hence, $y^b\in Z_\R$ by Lemma~\ref{lem:eventually-constant-template}. Since $\tom\in L_b$, we have $(y^b)_\tom=0$, so $y^b\in Y_+$.

If $s\in L_b$, then $s\in R$ and $(u_m^R)_s=0=(y^b)_s$. If $s\notin L_b$, then $(y^b)_s=1$ and $u_m^R$ takes only the values $0$ and $1$. Thus, $u_m^R\leqslant y^b$ for all $m$, so $\Phi(T)$ is order bounded.
\end{proof}

\begin{prop}\label{prop:upper-WF}
\leanformalization{\leanstatement{ThmBC/Upper/Reduction.lean}{94}{upper_WF_template}}
If $T\in\WF$, then $\Phi(T)\in Y^{\uparrow\infty}$.
\end{prop}

\begin{proof}
Let $R=\widehat T$. By Lemma~\ref{lem:fattening}, $R$ is well-founded. Suppose, towards a contradiction, that there is $y\in Y_+$ with
\[
u_m^R\leqslant y
\qquad(m\in\N).
\]
We prove by well-founded induction on $R$ that
\[
y_q\geqslant1
\qquad(q\in R).
\]
Fix $q\in R$ and assume the assertion has been proved for every immediate successor $q\concat n\in R$. Let $n\in\N$. If $q\concat n\in R$, the induction hypothesis gives $y_{q\concat n}\geqslant1$. If $q\concat n\notin R$, then the first exit from $R$ occurs at $q\concat n$. For every $j\in\N$, choosing $m\geqslant j$ gives
\[
(u_m^R)_{q\concat n\concat j}=(a_m)_j=1,
\]
and hence $y_{q\concat n\concat j}\geqslant1$. Lemma~\ref{lem:child-convergence-template}, applied at $q\concat n$, yields $y_{q\concat n}\geqslant1$.

Thus, $y_{q\concat n}\geqslant1$ for all $n$. Applying Lemma~\ref{lem:child-convergence-template} at $q$ gives $y_q\geqslant1$. This completes the induction. At the root, we get $y_\tom\geqslant1$, contradicting the defining condition $y_\tom=0$ for $Y$. Therefore, $\Phi(T)$ is not order bounded.
\end{proof}

\begin{thm}\label{thm:upper-complete}
\leanformalization{\leanstatement{ThmBC/Upper/Upper.lean}{350}{upper_complete_template}}
There is a separable Banach lattice $Y$ for which $Y^{\uparrow\infty}$ is $\mathbf{\Pi}^1_1$-complete.
\end{thm}

\begin{proof}
For the separable Banach lattice $Y=\{y\in Z_\R\del y_\tom=0\}$ constructed above, the continuous map $\Phi:\Tr\to Y^\N$ satisfies
\[
T\in\WF
\quad\equi\quad
\Phi(T)\in Y^{\uparrow\infty}
\]
by Propositions~\ref{prop:upper-IF} and~\ref{prop:upper-WF}. Hence, $\WF\leq_{\rm c}Y^{\uparrow\infty}$. Proposition~\ref{prop:coanalytic} gives coanalyticity, and $\mathbf{\Pi}^1_1$-completeness follows.
\end{proof}

\end{document}